\documentclass[11pt]{amsart}
\usepackage{amsmath,amsthm,amsfonts,amssymb,latexsym,epsfig}

\newtheorem{theorem}{Theorem}[section]

\newtheorem{lemma}{Lemma}[section]

\newtheorem{cor}{Corollary}[section]

\numberwithin{equation}{section}

\theoremstyle{definition}

\theoremstyle{remark}

\begin{document}
\title{On Weighted Remainder Form of Hardy-type Inequalities}
\author{Peng Gao}
\address{Division of Mathematical Sciences, School of Physical and Mathematical Sciences,
Nanyang Technological University, 637371 Singapore}
\email{penggao@ntu.edu.sg}
\date{July 30, 2009.}


\begin{abstract}
  We use different approaches to study a generalization of a result of Levin and Ste\v ckin concerning an inequality analogous to Hardy's
  inequality. Our results lead naturally to the study of weighted remainder form of Hardy-type
  inequalities.
\end{abstract}

\maketitle
\section{Introduction}
\label{sec 1} \setcounter{equation}{0}

  Let $p>1$ and $l^p$ be the Banach space of all complex sequences ${\bf a}=(a_n)_{n \geq 1}$. The celebrated
   Hardy's inequality \cite[Theorem 326]{HLP} asserts that for
   $p>1$ and any ${\bf a} \in l^p$,
\begin{equation}
\label{eq:1} \sum^{\infty}_{n=1}\Big{|}\frac {1}{n}
\sum^n_{k=1}a_k\Big{|}^p \leq \Big(\frac {p}{p-1}
\Big)^p\sum^\infty_{k=1}|a_k|^p.
\end{equation}

  Hardy's inequality can be regarded as a special case of the
   following inequality:
\begin{equation*}
   \Big | \Big |C \cdot {\bf a}\Big | \Big |^p_p =\sum^{\infty}_{n=1} \Big{|}\sum^{\infty}_{k=1}c_{n,k}a_k
    \Big{|}^p \leq U_p \sum^{\infty}_{n=1}|a_n|^p,
\end{equation*}
   in which $C=(c_{n,k})$ and the parameter $p>1$ are assumed
   fixed, and the estimate is to hold for all complex
   sequences ${\bf a} \in l^p$. The $l^{p}$ operator norm of $C$ is
   then defined as
\begin{equation*}
\label{02}
    ||C||_{p,p}=\sup_{||{\bf a}||_p \leq 1}\Big | \Big |C \cdot {\bf a}\Big | \Big |_p.
\end{equation*}
    Hardy's inequality thus asserts that the Ces\'aro matrix
    operator $C=(c_{j,k})$, given by $c_{j,k}=1/j , k\leq j$ and $0$
    otherwise, is bounded on {\it $l^p$} and has norm $\leq
    p/(p-1)$. (The norm is in fact $p/(p-1)$.) Hardy's inequality
    leads naturally to the study on $l^p$ norms of general
    matrices. For example, we say a matrix $A=(a_{j,k})$ is a weighted
    mean matrix if its entries satisfy $a_{j,k}=0, k >j$ and
\begin{equation*}
    a_{j,k}=\lambda_k/\Lambda_j,  ~~ 1 \leq k \leq
    j; \Lambda_j=\sum^j_{i=1}\lambda_i, \lambda_i \geq 0, \lambda_1>0.
\end{equation*}
   There are many studies on the $l^{p}$ operator norm of a weighted mean matrix and we refer the reader to the articles \cite{B4}-\cite{Be1},
\cite{G}-\cite{G5} and the references therein for more results in
this area.

    In this paper, we are interested in the following analogue of
    Hardy's inequality, given as Theorem 345 of \cite{HLP}, which asserts that the following inequality holds for
$0<p<1$ and $a_n \geq 0$ with $c_p=p^p$:
\begin{equation}
\label{1}
  \sum^{\infty}_{n=1}\Big( \frac 1{n} \sum^{\infty}_{k=n}a_k \Big
  )^p \geq c_p \sum^{\infty}_{n=1}a^p_n.
\end{equation}
   It is noted in \cite{HLP} that the constant $c_p=p^p$ may not be best possible
  and a better constant was indeed obtained by Levin and Ste\v ckin
  \cite[Theorem 61]{L&S}. Their result is more general as they proved, among other things, the
   following inequality (\cite[Theorem 62]{L&S}), valid for $0< p
   \leq 1/3, r \leq p$ or $1/3<p<1, r \leq 1-2p$ (note that this is given in \cite{L&S} as $1/3<p<1, r \leq (1-p)^2/(1+p)$ but an inspection of the proof of
   Theorem 62 in \cite{L&S} shows that they actually proved their result for $1/3<p<1, r \leq 1-2p$, see especially the proof of Lemma 3 in the proof of Theorem 62 in \cite{L&S}
   for this) with $a_n \geq 0$,
\begin{equation}
\label{4.1}
  \sum^{\infty}_{n=1}\frac 1{n^r} \Big( \sum^{\infty}_{k=n}a_k \Big
  )^p \geq c_{p,r} \sum^{\infty}_{n=1}\frac {a^p_n}{n^{r-p}},
\end{equation}
   where the constant $c_{p,r}=(p/(1-r))^p$ is best possible (see for example, \cite{G9}). It follows that inequality \eqref{1} holds
   for $0<p \leq 1/3$ with the best possible constant $c_p=(p/(1-p))^p$.

   The above result of Levin and Ste\v ckin has been studied in
   \cite{G6} and \cite{G9}. In \cite{G6}, a simple proof of
   inequality \eqref{4.1} for the case $0<r=p \leq 1/3$ is given.
   In \cite{G9}, inequality \eqref{4.1} is shown to hold for
   $0<r=p \leq 0.346$.

   It is our goal in this paper to first generalize the above result of Levin and Ste\v ckin.
   We make a convention in this paper that for any integer $k \geq 1$, $((k+1)^0-k^0)/0=\ln ((k+1)/k)$ and we
   shall prove in Section \ref{sec 3} the following
\begin{theorem}
\label{thm1}
   Let $a_n > 0, 0<p<1$. The following inequality holds for any number $r$ satisfying $(2+r)p \leq 1$,
\begin{equation}
\label{2.02}
   \sum^{\infty}_{n=1}\Big ( \frac {1}{n^r}\sum^{\infty}_{k=n}\Big(\frac {(k+1)^{r}-k^{r}}{r}\Big )a_k \Big
   )^p \geq \Big ( \frac {p}{1-rp}\Big
   )^p\sum^{\infty}_{n=1}a^p_n.
\end{equation}
    The above inequality reverses when $p \geq 1, 1/p-2 \leq r < 1/p$ or $p<0$, $1/p-2 \leq r < 1/p$. The constant is best possible.
\end{theorem}

     One can show following the construction in \cite{G9} that the
    constant in \eqref{2.02} is best possible. We let $q$ be the number defined by $1/p+1/q=1$ and note that by
the duality principle (see \cite{M}),
    the statement of Theorem \ref{thm1} is equivalent to the
    following
\begin{theorem}
\label{thm2}
   Let $a_n > 0, 0<p<1$. The following inequality holds for any number $r$ satisfying $(2+r)p \leq 1$,
\begin{equation}
\label{1.05}
   \sum^{\infty}_{n=1}\Big ( \Big ( \frac {(n+1)^{r}-n^{r}}{r}\Big )\sum^{n}_{k=1}\frac {a_k}{k^{r}} \Big
   )^q \leq \Big ( \frac {p}{1-rp}\Big
   )^{q}\sum^{\infty}_{n=1}a^q_n.
\end{equation}
    The above inequality also holds when $p > 1, 1/p-2 \leq r < 1/p$  and the reversed
    inequality \eqref{1.05}
    holds when $p<0$, $1/p-2 \leq r < 1/p$. The constant is best possible.
\end{theorem}

    We now write $r=\alpha +\beta/p$ in Theorem \ref{thm1} and note
    that for $\beta \leq 0$, we have
\begin{equation*}
    n^{\beta/p}\Big( \frac {(n+1)^{\alpha}-n^{\alpha}}{\alpha} \Big
) \geq  \Big( \frac
{(n+1)^{\alpha+\beta/p}-n^{\alpha+\beta/p}}{\alpha+\beta/p} \Big
).
\end{equation*}
    This combined with inequality \eqref{2.02} allows us to deduce
    the following (via a change of variables $a_n \rightarrow
    n^{-\beta/p}a_n$)
\begin{cor}
\label{cor0}
   Let $a_n \geq 0, \beta \leq 0< \alpha, 0<p<1$. The following inequality holds for $0 < p \leq
   (1-\beta)/(2+\alpha)$,
\begin{equation*}
   \sum^{\infty}_{n=1}\frac {1}{n^{\beta}}\Big ( \frac {1}{n^{\alpha}}\sum^{\infty}_{k=n}\Big((k+1)^{\alpha}-k^{\alpha}\Big )a_k \Big
   )^p \geq \Big ( \frac {\alpha p}{1-\beta-\alpha p}\Big
   )^p\sum^{\infty}_{n=1}\frac {a^p_n}{n^{\beta}}.
\end{equation*}
    The constant is best possible.
\end{cor}

    One can also deduce the cases $0< p
   \leq 1/3, r \leq p$ or $1/3<p<1, r \leq 1-2p$ of inequality
   \eqref{1} via similar transformations of inequality
   \eqref{2.02}.

     The case $r=0$ in Theorem \ref{thm1} implies
    the following
\begin{cor}
\label{cor1}
   Let $a_n \geq 0$. For $0 < p \leq 1/2$, we have
\begin{equation*}
  \sum^{\infty}_{n=1}\Big( \sum^{\infty}_{k=n} \ln \Big(\frac {k+1}{k}\Big ) a_k \Big
  )^p \geq p^p \sum^{\infty}_{n=1}a^p_n.
\end{equation*}
    The constant is best possible.
\end{cor}

    Note that as $\ln ((k+1)/k) \leq 1/k$, Corollary \ref{cor1}
    implies the following well-known Copson's inequality \cite[Theorem 344]{HLP} when $0<p
    \leq 1/2$:
\begin{equation*}
  \sum^{\infty}_{n=1}\Big( \sum^{\infty}_{k=n} \frac {a_k}{k} \Big
  )^p \geq p^p \sum^{\infty}_{n=1}a^p_n.
\end{equation*}

   Similarly, the case $r=0$ in Theorem
   \ref{thm2} implies the following
\begin{cor}
\label{cor2}
   Let $a_n > 0$. For $-1 \leq  p <0$, we have
\begin{equation*}
   \sum^{\infty}_{n=1}\Big( \ln \Big(\frac {n+1}{n} \Big )\sum^{n}_{k=1}a_k \Big
   )^p \leq \Big (\frac {p}{p-1}\Big )^{p}\sum^{\infty}_{n=1}a^p_n.
\end{equation*}
    The constant is best possible.
\end{cor}

    We point out here that Corollary \ref{cor2} implies the well-known
    Knopp's inequality \cite[Satz IV]{K} (which is inequality \eqref{eq:1} with $p<0$ and $a_n>0$) when $-1 \leq p <0$.

   We note that it is pointed out in \cite{G9} that inequality
   \eqref{1} can not hold for all $0<p<1$ with the constant $c_p$
   being $(p/(1-p))^p$.
   However, Levin and Ste\v
   ckin \cite[Theorem 61]{L&S} was able to improve the constant $c_p=p^p$ for all $0<p<1$ as their result is given in the
   following:
\begin{theorem}
\label{thm4}
    Inequality \eqref{1} holds with $c_p$ being
\begin{equation*}
    c_p = \left\{ \begin{array}{ll}
\Big (\frac
{p}{1-p}\Big )^p, & 0<p \leq 1/3;\\
\frac {1}{2}\Big (\frac {1+p}{1-p} \Big
    )^{1-p}, & 1/3 < p \leq
3/5; \\
2\Big ( \frac {p}{3-p}\Big )^p, & 3/5 \leq
    p<1.\end{array}
\right.
\end{equation*}
\end{theorem}

     Our method in this paper
   allows us to give another proof of the above result. In fact,
   we shall prove the following result in Section \ref{sec 4}:
\begin{theorem}
\label{thm5}
    Let $0<p<1, 0<r \leq p$. Inequality \eqref{4.1} holds with $c_{p,r}$ with
\begin{equation*}
    c^{-1}_{p,r}=(2-p+r)\Big ( \frac {1-p}{1-p+2r} \Big )^{1-p}.
\end{equation*}
\end{theorem}

     We note here the constant $c^{-1}_{p,r}$ in the statement of
     Theorem \ref{thm5} is nothing but the constant $\chi(r)$ defined in
     Lemma 5 in the proof of Theorem 62 in \cite{L&S}. We now say
     a few words on how to deduce Theorem \ref{thm4} from Theorem
     \ref{thm5}, this is also given in Theorem 62 of \cite{L&S}.
     First, it is easy to show that for fixed $p$, $c^{-1}_{p,r}$
     is minimized at $r=(3-2p)(1-p)/2p$. When $1/3 \leq p \leq
     3/5$, we have $p \leq (3-2p)(1-p)/2p$, hence on setting $r=p$
     in Theorem \ref{thm5} implies the corresponding cases of Theorem
     \ref{thm4}. When $3/5 \leq p <1$, we have $p \geq (3-2p)(1-p)/2p$ and a combination of
     Lemma 7 in the proof of Theorem 62 in \cite{L&S} and setting $r = (3-2p)(1-p)/2p$ in Theorem \ref{thm5} implies the
corresponding cases of Theorem
     \ref{thm4}.


    Our method in proving Theorem \ref{thm4} and Theorem \ref{thm5} is more flexible and there is still room to further improve
    the constant $c_p$ or $c_{p,r}$, when they are not best possible. In this paper, we shall only consider the constant $c_{1/2}$ of the special case $p=1/2$ in
    \eqref{1}. It's given as $1/\sqrt{2}$ in \cite{HLP} and was
    improved to be $\sqrt{3}/2$ by Levin and Ste\v ckin in
    \cite{L&S}. The author has shown in \cite{G6} that one can
    take $c_{1/2}=0.8967$ but at that time he was not aware that
    Boas and de Bruijn \cite{B&dB} showed that $15/17 \approx 0.8824 < c_{1/2}
    <1/1.08 \approx 0.9259$ and De Bruijn \cite{deB} showed that
    $c_{1/2} \approx 1/1.1064957714 \approx 0.90375$. With less
    effort than de Bruijn's analysis in \cite{deB}, we shall use our approach in this
    paper to show in Section \ref{sec 4} that we can take $c_{1/2}$ to be $0.9$, which coincides with the optimal $c_{1/2}$ for the
    first two decimal expansions. This is given in the
    following
\begin{theorem}
\label{thm6}
    Inequality \eqref{2.02} holds when $p=1/2$ with $c_{1/2}=0.9$.
\end{theorem}


   We note the following result:
\begin{theorem}
\label{thm7}
    Let $a_n>0$ and $\alpha > 0$. Then for $p<0$ or $p \geq 1$ and $\alpha p > 1$,
    we have
\begin{equation}
\label{1.5}
   \sum^{\infty}_{n=1}\Big{(}\frac {1}{n^{\alpha}} \sum^n_{k=1}(k^{\alpha}-(k-1)^{\alpha})a_k\Big{)}^p
\leq \Big(\frac {\alpha p}{\alpha p-1}
\Big)^p\sum^\infty_{k=1}a^p_k.
\end{equation}
    The constant is best possible.
\end{theorem}
     The special case $p>1, \alpha \geq 1, \alpha p >1$ of
    inequality \eqref{1.5}
    was proved by the author in \cite{G}. The general cases
    of inequality \eqref{1.5}
    were proved by Bennett in \cite{Be1}.

    It's easy to show that we
     have, for $\alpha>0, r \geq 1$, $p \geq 1$ that when $k \geq 1$,
\begin{equation*}
    k^{(1-r)/p}\Big
   (\frac {k^{\alpha}-(k-1)^{\alpha}}{\alpha} \Big ) \leq  \frac
{k^{\alpha+(1-r)/p}-(k-1)^{\alpha+(1-r)/p}}{\alpha+(1-r)/p}.
\end{equation*}
    The above inequality reverses when $p<0$. Replacing $\alpha$
    with $\alpha+(1-r)/p$ in \eqref{1.5} and applying the above
    inequality, we deduce immediately the following (via a change
    of variables $a_n \rightarrow n^{(r-1)/p}a_n$) result (\cite[Theorem
    1]{Be1}):
\begin{cor}
\label{cor7}
    Let $a_n>0$. Suppose that $\alpha > 0$ and $r \geq 1$. Then for $p<0$ or $p \geq 1$ and $\alpha p > r$,
    we have
\begin{equation*}
   \sum^{\infty}_{n=1}n^{r-1}\Big{(}\frac {1}{n^{\alpha}} \sum^n_{k=1}(k^{\alpha}-(k-1)^{\alpha})a_k\Big{)}^p
\leq \Big(\frac {\alpha p}{\alpha p-r}
\Big)^p\sum^\infty_{k=1}k^{r-1}a^p_k.
\end{equation*}
    The constant is best possible.
\end{cor}

     We point out here that we will present two proofs of Theorem \ref{thm1} in
    Section \ref{sec 3}. The first one can be viewed as an analogue to Bennett's proof of Theorem
    \ref{thm7} and the second one is a
    generalization of the proof of inequality \eqref{4.1} given in \cite{L&S}. One then
    asks whether one can adapt the approach used in the second proof of Theorem \ref{thm1} to give another proof
    of Theorem \ref{thm7} and this is indeed
    possible as we will give an alternative proof of Theorem
    \ref{thm7} in Section \ref{sec 5}.

    As it is pointed out in \cite{Be1} that inequality \eqref{1.5} fails to hold when $\alpha p \leq
     1$. One therefore wonders whether there are any analogues of
     inequality \eqref{1.5} that hold when $\alpha p \leq 1$. For this we note
     that it follows from Theorem \ref{thm1} that the reversed inequality \eqref{2.02}
    holds when $p \geq 1$ or $p<0$ under certain restrictions on
    $r$. One may view these reversed inequalities as the $\alpha p \leq 1$ analogues to inequality
    \eqref{1.5}. However, the duality principle also allows one to interpret these inequalities as $\alpha p>1$ (with a different $\alpha$)
    analogues to \eqref{1.5}. To see this, we take the $p > 1$ case in Theorem
    \ref{thm1} as an example and we use its dual version,
    Theorem \ref{thm2} with $p>1, 1/p-2 \leq r <1/p$. We interchange the variables $p$ and $q$ and replace $r$ by
    $-r$ to recast inequality \eqref{1.05} for this case as ($1/p-1 < r \leq
   1+1/p$):
\begin{equation*}
     \sum^{\infty}_{n=1}\Big ( \Big ( \frac {n^{-r}-(n+1)^{-r}}{r}\Big )\sum^{n}_{k=1}k^{r}a_k \Big
   )^p \leq \Big ( \frac {p}{(r+1)p-1}\Big
   )^{p}\sum^{\infty}_{n=1}a^p_n.
\end{equation*}
     Note that the above inequality is analogue to inequality \eqref{1.5} in the sense that
     we have $(r+1)p>1$ here. We can further recast the above
     inequality as
\begin{equation}
\label{1.7}
     \sum^{\infty}_{n=1}\Big ( \Big ( \frac {n^{-r}-(n+1)^{-r}}{r}\Big
     )(1+r)\Big(\sum^n_{i=1}i^r \Big)
     \frac {1}{\sum^n_{i=1}i^r}\sum^{n}_{k=1}k^{r}a_k \Big
   )^p \leq \Big ( \frac {(r+1)p}{(r+1)p-1}\Big
   )^{p}\sum^{\infty}_{n=1}a^p_n.
\end{equation}
     The above inequality and inequality \eqref{bound1} below
     imply immediately the following inequality for $a_n>0, p>1, 2 \leq \alpha \leq 2+1/p$:
\begin{equation}
\label{8}
   \sum^{\infty}_{n=1}\Big{(}\frac
1{\sum^n_{i=1}i^{\alpha-1}}\sum^n_{i=1}i^{\alpha-1}a_i\Big{)}^p
\leq  \Big(\frac {\alpha p}{\alpha p-1} \Big
)^p\sum^{\infty}_{n=1}a^p_n.
\end{equation}

     The above inequality has been studied in \cite{G}, \cite{Be1}, \cite{G6}, \cite{G8} and \cite{G5}. The author \cite{G} and Bennett \cite{Be1}
     proved inequality \eqref{8} for $p>1, \alpha \geq 2$ or $0< \alpha \leq 1, \alpha p
>1$
     independently. The author \cite{G6} has shown that \eqref{8} holds for $p \geq 2, 1 \leq
   \alpha \leq 1+1/p$ or $1 < p \leq 4/3, 1+1/p \leq
   \alpha \leq 2$. Recently, the author \cite{G5} has shown that inequality
   \eqref{8} holds for $p \geq 2$, $0 \leq \alpha \leq 1$. In
   \cite[Corollary 2.4]{G8}, it is shown that inequality \eqref{8}
   holds for $\alpha>0, p<0$.

     Other than the above point of view of the reversed inequality of
    \eqref{2.02} using the duality principle, we may also regard the (reversed)
     inequality of
    \eqref{2.02} as a type of ``weighted remainder form of
    Hardy-type
    inequalities", a terminology we use after Pe\v cari\'c and Stolarsky, who studied a special case of this type of inequalities in
    \cite[Sec 3]{P&S}. Theorem \ref{thm1} thus leads naturally to the study of the
    following weighted remainder form of
    Hardy-type inequalities in general:
\begin{equation}
\label{6.1}
    \sum^{\infty}_{n=1}\Big (\sum^{\infty}_{k=n}\frac
   {\lambda_ka_k}{\Lambda_n}\Big )^p \leq \Big (\frac {p}{p-L}
   \Big )^p\sum^{\infty}_{n=1}a^p_n,
\end{equation}
   where $(\lambda_n)$ is a positive sequence satisfying
   $\Lambda_n=\sum^{\infty}_{k=n}\lambda_k<+\infty$ and $L$ is a number such that $L<p$ when $p>0$ and $L>p$ when
   $p<0$. We want the above
inequality to hold for $p > 1$ or
   $p<0$ and any positive sequence $(a_n)$ satisfying $\sum^{\infty}_{n=1}a^p_{n}<+\infty$.
   We also want the reversed inequality of \eqref{6.1} to hold when $0<p<1$. We shall study inequality \eqref{6.1} in Section \ref{sec 6}.
   We shall find conditions on the $\lambda_n$'s
   so that inequality \eqref{6.1} (or its reverse) can hold under
   these conditions.

\section{A heuristic approach to inequality \eqref{2.02}}
\label{sec 2} \setcounter{equation}{0}

   In this section, we first give a heuristic approach towards establishing inequality \eqref{2.02}. This approach will provide motivation and
   serve as a guideline for our proof of Theorem \ref{thm1} later. In fact, the approach we discuss here is in
some sense a ``natural" approach towards establishing Hardy-type
inequalities. For simplicity, we consider inequality \eqref{2.02}
for $0<p<1, r>0, rp <1$. A general approach towards establishing
the above inequality is to apply the reversed H\"older's
    inequality to get
\begin{equation}
\label{2.2'}
  \Big(\sum^{\infty}_{k=n}\Big((k+1)^{r}-k^{r}\Big
)a_k \Big )^p \geq
\Big(\sum^{\infty}_{k=n}W^{1/(1-p)}_k\Big)^{p-1}\Big(\sum^{\infty}_{k=n}W_k\big((k+1)^{r}-k^{r}\big
)^pa^p_k \Big),
\end{equation}
   where $(W_k)$ is a sequence to be determined. A general
   discussion on Hardy-type inequalities in
   \cite{G5} implies that one can in fact obtain the best possible constant on choosing $W_k$ properly.

   We now give a description of one choice for the $W_k$'s in \eqref{2.2'}. In fact, more naturally,
   we write
   $\big((k+1)^{r}-k^{r}\big )a_k=\big((k+1)^{r}-k^{r}\big )k^{-\gamma}k^{\gamma}a_k$,
    so that
   by the reversed H\"older's inequality (and one can reconstruct the $W_k$'s from this), we have
\begin{equation*}
  \Big(\sum^{\infty}_{k=n}\big((k+1)^{r}-k^{r}\big )a_k \Big )^p \geq
\Big(\sum^{\infty}_{k=n}\big((k+1)^{r}-k^{r}\big
)^{p/(p-1)}k^{-\gamma p
/(p-1)}\Big)^{p-1}\Big(\sum^{\infty}_{k=n}k^{\gamma
   p}a^p_k \Big),
\end{equation*}
   where $\gamma <r-1/p<0$ (this guarantees the finiteness of the two factors of the right-hand side expressions above) is a parameter to be chosen later.
   Using this, we
   then have
\begin{eqnarray*}
 &&  \sum^{\infty}_{n=1}\Big ( \frac {1}{n^{r}}\sum^{\infty}_{k=n}\Big(\frac{(k+1)^{r}-k^{r}}{r}\Big
)a_k \Big
   )^p  \\
 &\geq &  \sum^{\infty}_{n=1}\frac
{\Big(\sum^{\infty}_{k=n}\big((k+1)^{r}-k^{r}\big
)^{p/(p-1)}k^{-\gamma p
/(p-1)}\Big)^{p-1}\Big(\sum^{\infty}_{k=n}k^{\gamma
   p}a^p_k \Big)}{r^pn^{rp}} \nonumber \\
   & =&  \sum^{\infty}_{k=1}a^p_kk^{\gamma p}\sum^{k}_{n=1}\frac
   {\Big(\sum^{\infty}_{i=n}\big((i+1)^{r}-i^{r}\big
)^{p/(p-1)}i^{-\gamma p/(p-1)}\Big)^{p-1}}{r^pn^{rp}}.
\end{eqnarray*}
   Asymptotically, we have
\begin{eqnarray}
\label{2.5}
 & &\sum^{\infty}_{i=n}\big((i+1)^{r}-i^{r}\big
)^{p/(p-1)}i^{-\gamma p/(p-1)} \\
&\sim & r^{p/(p-1)}
\sum^{\infty}_{i=n}i^{(r-1-\gamma) p/(p-1)} \nonumber \\
&\sim & \frac
   {r^{p/(p-1)}n^{(r-1-\gamma) p/(p-1)+1}}{(\gamma+1-r)
   p/(p-1)-1}. \nonumber
\end{eqnarray}
   It follows that asymptotically, we have
\begin{eqnarray*}
  && k^{\gamma p}\sum^{k}_{n=1}\frac
   {\Big(\sum^{\infty}_{i=n}\big((i+1)^{r}-i^{r}\big
)^{p/(p-1)}i^{-\gamma p/(p-1)}\Big)^{p-1}}{r^pn^{rp}}   \\
   &\sim & \frac {1}{((\gamma+1-r) p/(p-1)-1)^{p-1}}k^{\gamma p}\sum^{k}_{n=1}\frac
   {1}{n^{1+\gamma p}} \\
   & \sim & -\frac {1}{((\gamma+1-r) p/(p-1)-1)^{p-1}}\frac {1}{\gamma p}.
\end{eqnarray*}
    We then want to choose $\gamma$ so that the last expression
    above is maximized and calculation shows that in this case we
    need to take $\gamma=(rp-1)/p^2 (< r-1/p) $ and the so taken
    $\gamma$ makes the value of the last expression above being
    exactly the constant appearing on the right-hand side of
    \eqref{2.02}.

    The above approach can be applied to discuss inequality
    \eqref{1.05} similarly and in this case, we can make our argument
    rigorous to give a proof of Theorem \ref{thm2}.

    Proof of Theorem \ref{thm2}:

    Due to the similarities of the proofs (taken into account the reversed inequality of \eqref{bound1}), we may assume $0<p<1$
    here. By the reversed H\"older's inequality, we have
\begin{equation*}
   \Big (\sum^n_{k=1}\frac {a_k}{k^{r}}\Big )^q \leq \Big
   ( \sum^n_{k=1}\frac {k^{\gamma}a^q_k}{k^{rq}} \Big
   )\Big(\sum^n_{k=1}k^{\gamma/(1-q)}\Big )^{q-1},
   \hspace{0.1in} \gamma=1/p+r/(p-1).
\end{equation*}
    It follows that
\begin{eqnarray*}
  && \sum^{\infty}_{n=1}\Big ( \Big(\frac {(n+1)^{r}-n^{r}}{r}\Big )\sum^{n}_{k=1}\frac {a_k}{k^{r}} \Big
   )^q  \\
   &\leq &
   \sum^{\infty}_{k=1}\frac
   {k^{\gamma}a^q_k}{k^{rq}}\sum^{\infty}_{n=k} \Big( \frac {(n+1)^{r}-n^{r}}{r} \Big
   )^q\Big ( \sum^n_{i=1}i^{\gamma/(1-q)} \Big
   ) ^{q-1}.
\end{eqnarray*}
    We now note the following inequality (\cite[Lemma 2, p. 18]{L&S}), which asserts for $r
    \geq 1$, we have
\begin{equation}
\label{bound1}
   \sum^n_{i=1}i^r \geq \frac {r}{1+r}\frac
   {n^r(n+1)^r}{(n+1)^r-n^r}.
\end{equation}
    The above inequality reverses when $-1<r \leq 1$ (only the case $r \geq 0$ of the above inequality was proved in \cite{L&S}
     but one checks easily that the proof extends to the case $r
     >-1$).

    When $\gamma/(1-q) \geq 1$, which is equivalent to
    the condition $(r+2)p \leq 0$, we can apply
    estimation \eqref{bound1} to get
\begin{equation}
\label{2.6}
    \sum^n_{i=1}i^{\gamma/(1-q)} \geq \frac {1}{1+\gamma/(1-q)}\Big
    (\int^{n+1}_nx^{-\gamma/(1-q)-1}dx\Big )^{-1}
\end{equation}

    This combines with \eqref{2.6} implies that
\begin{eqnarray*}
  && \sum^{\infty}_{n=1}\Big ( \Big(\frac {(n+1)^{r}-n^{r}}{r}\Big )\sum^{n}_{k=1}\frac {a_k}{k^{r}} \Big
   )^q  \\
   &\leq & \Big ( \frac {1}{1+\gamma/(1-q)}\Big )^{q-1}
   \sum^{\infty}_{k=1}\frac
   {k^{\gamma}a^q_k}{k^{rq}}\sum^{\infty}_{n=k}\Big ( \int^{n+1}_nx^{r-1}dx\Big )^q\Big ( \int^{n+1}_nx^{-\gamma/(1-q)-1}dx \Big
   )^{1-q} \\
   & \leq & \Big ( \frac {1}{1+\gamma/(1-q)}\Big )^{q-1}
   \sum^{\infty}_{k=1}\frac
   {k^{\gamma}a^q_k}{k^{rq}}\sum^{\infty}_{n=k}\int^{n+1}_nx^{q(r-1)+(1-q)(-\gamma/(1-q)-1)}dx \\
   &=& \Big ( \frac {1}{1+\gamma/(1-q)}\Big )^{q}=\Big ( \frac { p}{1-r p}\Big
   )^{q}.
\end{eqnarray*}
    This completes the proof of Theorem \ref{thm2}.
\section{Proof of Theorem \ref{thm1}}
\label{sec 3} \setcounter{equation}{0}

    We shall give two proofs of the case $0<p<1$ here and as we mentioned earlier, the first proof can be
    viewed as an analogue to Bennett's proof (\cite[Theorem 1]{Be1}) of Theorem \ref{thm7} and the second proof is a
   generalization of the proof of Theorem 62 in \cite{L&S}. An inspection of the
    proofs shows that they also work for the cases $p \geq 1$ and $p<0$ as well (taken into account the reversed inequality of \eqref{bound1}). The
   first proof given below can also be viewed as a translation
   of the proof of Theorem \ref{thm2} given in the previous
   section via duality. From now on in this section, we assume
   $0<p<1$.

   The first proof:

   Our discussion in Section \ref{sec 2} suggests that if we take the approach there,
   then we should take an auxiliary sequence $(W_k)$ so that asymptotically, a similar expression would lead to something like the last expression of \eqref{2.5}.
   One can see in what follows that our selection of the auxiliary sequence in the proof is then guided by this.  By the reversed H\"older's inequality, we have
\begin{eqnarray}
\label{3.1}
 && \Big(\sum^{\infty}_{k=n}\Big(\frac {(k+1)^{r}-k^{r}}{r}\Big
)a_k \Big )^p  \\
& \geq &
\Big(\sum^{\infty}_{k=n}\big(k^{r-1/p}-(k+1)^{r-1/p}\big)\Big)^{p-1}
\cdot \Big(\sum^{\infty}_{k=n}\Big(\frac {(k+1)^{r}-k^{r}}{r}\Big
)^p\big(k^{r-1/p}-(k+1)^{r-1/p}\big)^{1-p}a^p_k
\Big) \nonumber \\
  &=&
  n^{(1-rp)(1-p)/p}\sum^{\infty}_{k=n}\Big(\frac {(k+1)^{r}-k^{r}}{r}\Big
)^p\big(k^{r-1/p}-(k+1)^{r-1/p}\big)^{1-p}a^p_k. \nonumber
\end{eqnarray}
   We then proceed as in Section \ref{sec 2} to see that
\begin{eqnarray*}
   && \sum^{\infty}_{n=1}\Big ( \frac {1}{n^{r}}\sum^{\infty}_{k=n}\Big(\frac {(k+1)^{r}-k^{r}}{r}\Big )a_k \Big
   )^p \\
   & \geq &  \sum^{\infty}_{k=1}a^p_k\Big(\frac {(k+1)^{r}-k^{r}}{r}\Big )^p\Big(k^{r-1/p}-(k+1)^{r-1/p}\Big)^{1-p}\sum^{k}_{n=1}n^{1/p-(1+r)}.
\end{eqnarray*}
   It therefore suffices to show that
\begin{equation}
\label{2.3}
   \Big(\frac {(k+1)^{r}-k^{r}}{r}\Big )^p\big(k^{r-1/p}-(k+1)^{r-1/p}\big)^{1-p}\sum^{k}_{n=1}n^{1/p-(1+r)}
   \geq \Big ( \frac {p}{1-rp}\Big
   )^p.
\end{equation}

    We now apply inequality \eqref{bound1} to see that in order for inequality
    \eqref{2.3} to hold, it suffices to show that for $n \geq 1$ (note that for $(2+r)p \leq 1, 1/p-(1+r) \geq 1$),
\begin{equation*}
   \Big(\frac {(n+1)^{r}-n^{r}}{r}\Big )^p\Big( \frac
{n^{r-1/p}-(n+1)^{r-1/p}}{(1-r p)/p} \Big)^{1-p} \geq
   \frac
   {n^{1+r-1/p}-(n+1)^{1+r-1/p}}{1/p-1-r}.
\end{equation*}

    We can recast the above inequality  as
\begin{equation}
\label{3.6}
   \Big(\int^{n+1}_nx^{r-1}dx \Big)^p\Big(\int^{n+1}_nx^{r-1/p-1}dx \Big)^{1-p} \geq
   \int^{n+1}_nx^{r-1/p}dx.
\end{equation}
   H\"older's inequality now implies the above inequality and
   this completes the first proof.

   The second proof:

   Similar to \eqref{2.2'}, we have
\begin{eqnarray*}
  && \sum^{\infty}_{n=1}\Big ( \frac {1}{n^{r}}\sum^{\infty}_{k=n}\Big( \frac {(k+1)^{r}-k^{r}}{r}\Big )a_k \Big
   )^p \\
   &\geq & \sum^{\infty}_{k=1}a^p_k\Big( \frac {(k+1)^{r}-k^{r}}{r}\Big )^pW_k\sum^k_{n=1}\frac {1}{n^{rp}}\Big (
\sum^{\infty}_{i=n}W^{1/(1-p)}_i \Big )^{p-1}.  \nonumber
\end{eqnarray*}

   We now choose $W_k$ to be
\begin{equation*}
    W_k=\Big( \frac {(k+1)^{r}-k^{r}}{r}\Big )^{-p}\Big ( \sum^k_{i=1}i^{\gamma} \Big )^{-1}, \hspace{0.1in}
\gamma=\frac {1-r p}{p}-1.
\end{equation*}
    Using \eqref{bound1} and \eqref{3.6}, we have
\begin{eqnarray*}
     W_k &\leq& (1+\gamma)\Big( \int^{k+1}_k x^{r-1}dx \Big
     )^{-p}\Big ( \int^{k+1}_k x^{-\gamma-1}dx \Big ) \\
     & \leq & (1+\gamma)\Big( \int^{k+1}_k x^{-\gamma-2}dx \Big
     )^{1-p}.
\end{eqnarray*}
     It follows that
\begin{eqnarray*}
   && \Big(\frac {(k+1)^{r}-k^{r}}{r}\Big )^pW_k\sum^k_{n=1}\frac {1}{n^{r p}}\Big (
\sum^{\infty}_{i=n}W^{1/(1-p)}_i \Big )^{p-1}  \\
 & \geq & (1+\gamma)^{-1}\Big ( \sum^k_{i=1}i^{\gamma} \Big
 )^{-1}\sum^k_{n=1}\frac {1}{n^{r p}}\Big (
\sum^{\infty}_{i=n}\int^{i+1}_i x^{-\gamma-2}dx \Big )^{p-1} \\
 &=& (1+\gamma)^{-p}=\Big ( \frac {p}{1-r p}\Big
   )^p.
\end{eqnarray*}
   This now completes the second proof of Theorem \ref{thm1}.

\section{Proofs of Theorem \ref{thm5} and Theorem \ref{thm6}}
\label{sec 4} \setcounter{equation}{0}

    We first give the proof of Theorem \ref{thm5} and we need a
    lemma:
\begin{lemma}
\label{lem1}
    Let $0<p<1, 0<r \leq p$, $\beta \geq 1+2r/(1-p)$. The function
\begin{equation*}
    f_{p,r,\beta}(x)=x^{-1}\Big((1+x)^{\beta-r/(1-p)}-(1+x)^{-r/(1-p)}
    \Big )
\end{equation*}
    is an increasing function of $0 \leq x \leq 1$.
\end{lemma}
\begin{proof}
   We have
   $f'_{p,r,\beta}(x)=x^{-2}(1+x)^{-r/(1-p)-1}h_{p,r,\beta}(x)$,
   where
\begin{equation*}
   h_{p,r,\beta}(x)=1+x-(1+x)^{\beta+1}+\frac {x}{1-p}\Big ((\beta(1-p)-r)(1+x)^{\beta}+r \Big
   ).
\end{equation*}
   We also have
\begin{equation*}
   \frac {1-p}{1-p+r}h'_{p,r,\beta}(x)=\beta \Big ( \frac {\beta(1-p)-r}{1-p+r}
   \Big)x(1+x)^{\beta-1}-\Big ( (1+x)^{\beta}-1 \Big ) \geq 0,
\end{equation*}
   where the last inequality above follows from the mean value
   theorem and our assumption on $\beta$. As $h_{p,r,\beta}(0)=0$, it follows that $h_{p,r,\beta}(x) \geq 0$ for $0
   \leq x \leq 1$. We then deduce from this that $f_{p,r,\beta}(x)$ is an
   increasing function of $0 \leq x \leq 1$ and this completes the
   proof.
\end{proof}

    We now return to the proof of Theorem \ref{thm5} and by a change of variables, $a_n \rightarrow n^{(r-p)/p}a_n$, we can recast
    inequality \eqref{4.1} as
\begin{equation}
\label{4.0}
  \sum^{\infty}_{n=1}\frac 1{n^r} \Big( \sum^{\infty}_{k=n}k^{(r-p)/p}a_k \Big
  )^p \geq c_{p,r} \sum^{\infty}_{n=1}a^p_n.
\end{equation}
    We follow the process in the first proof of Theorem \ref{thm1} in Section \ref{sec 3}, but
    this time, instead of using
    $k^{r-1/p}-(k+1)^{r-1/p}$ in \eqref{3.1}, we
    use $k^{-\beta}-(k+1)^{-\beta}$, where $\beta>0$ is a constant to be chosen later. The
    same process then leads to inequality \eqref{4.0} with the
    constant $c_{p,r}$ given by
\begin{equation}
\label{4.02}
  \min_{k \geq 1} \Big (  \big(k^{-\beta}-(k+1)^{-\beta}\big)^{1-p}k^{r-p}\sum^k_{n=1}n^{\beta(1-p)-r} \Big ).
\end{equation}
    We note the following
    inequality (\cite[Lemma 1, p. 18]{L&S}), which asserts for $0 \leq r
    \leq 1$, we have
\begin{equation}
\label{4.2}
   \sum^n_{i=1}i^r \geq \frac {n(n+1)^r}{1+r}.
\end{equation}

    We now assume $r/(1-p) \leq \beta \leq (1+r)/(1-p)$ so that $0
    \leq \beta(1-p)-r \leq 1$ and we can use the bound \eqref{4.2}
    in \eqref{4.02} to see that
\begin{equation*}
    c_{p,r} \geq \min_{k \geq 1}\frac
    {f^{1-p}_{p,r,\beta}(1/k)}{1+\beta(1-p)-r},
\end{equation*}
    where $f_{p,r,\beta}(x)$ is defined as in Lemma \ref{lem1}. We now take $\beta=1+2r/(1-p)$ and it is easy to verify that
    the so chosen $\beta$ satisfies $r/(1-p) \leq \beta \leq (1+r)/(1-p)$.
    It then follows from Lemma \ref{lem1} that $\min_{k \geq
   1}f^{1-p}(1/k)=\lim_{x \rightarrow
   0^+}f^{1-p}(x)=\beta^{1-p}=(1+2r/(1-p))^{1-p}$. This now leads to the
   constant $c_{p,r}$ given in the statement of Theorem \ref{thm5}
   and this completes the proof.


   We now give the proof of Theorem \ref{thm6}. Again we follow the process in the first proof of Theorem \ref{thm1} in Section \ref{sec 3} and similar to
   our proof of Theorem \ref{thm5} above, instead of using
    $k^{r-1/p}-(k+1)^{r-1/p}$ in \eqref{3.1}, we
    use $k^{-\beta}-(k+1)^{-\beta}$, where $\beta>0$ is a constant to be determined. The
    same process then leads to inequality \eqref{1} with
    $c_{1/2}=\min_{k \geq 1}s_k$, where
\begin{equation*}
    s_k= \big(k^{-\beta}-(k+1)^{-\beta}\big)^{1/2}\sum^k_{n=1}n^{(\beta-1)/2} .
\end{equation*}
    We want to choose $\beta$ properly to maximize $c_{1/2}$.  On considering $s_1$ and
    $\lim_{k \rightarrow +\infty}s_k$, we see that
\begin{equation*}
    c_{1/2} \geq \min \Big( (1-2^{-\beta})^{1/2}, \frac
    {2\beta^{1/2}}{1+\beta} \Big).
\end{equation*}
    Note that when $\beta=2$, $(1-2^{-2})^{1/2}=\sqrt{3}/2$ and
    when $\beta=3$, $2*3^{1/2}/(1+3)=\sqrt{3}/2$. As
    $(1-2^{-\beta})^{1/2}$ is an increasing function of $\beta$
    while $2\beta^{1/2}/(1+\beta)$ is a decreasing function of
    $\beta \geq 1$, our calculations above show that it suffices to
    consider $2 \leq \beta \leq 3$. On setting
    $(1-2^{-\beta})^{1/2}=2\beta^{1/2}/(1+\beta)$, we find that
    the optimal $\beta$ is approximately $2.4739$ and the value of
    $(1-2^{-\beta})^{1/2}$ or $2\beta^{1/2}/(1+\beta)$ at this
    number is approximately $0.9055$. This suggests that in order
    to maximize the value of $c_{1/2}$, we need to take $\beta$ to
    be around $2.47$. We now take $\beta=2.4$ instead and use the bound \eqref{4.2} to see that
\begin{equation*}
    \big(k^{-\beta}-(k+1)^{-\beta}\big)^{1/2}\sum^k_{n=1}n^{(\beta-1)/2}
    \geq \frac {2}{1+\beta}u^{1/2}_{\beta}(1/k),
\end{equation*}
    where
\begin{equation*}
    u_{\beta}(x)=x^{-1}\Big((1+x)^{\beta-1}-(1+x)^{-1} \Big ).
\end{equation*}
    We have $u'_{\beta}(x)=x^{-2}(1+x)^{-2}v_{\beta}(x)$, where
\begin{equation*}
    v_{\beta}(x)=1+x-(1+x)^{1+\beta}+x(1+(\beta-1)(1+x)^{\beta}).
\end{equation*}
    It's easy to check that $v'_{\beta}(0)=0$ and that
\begin{equation*}
     v''_{\beta}(x)=\beta(1+x)^{\beta-2}(\beta-3+(\beta^2-\beta-2)x).
\end{equation*}
    It's also easy to see that the last factor of the right-hand
    expression above is $<0$ when $\beta=2.4$ and $0 \leq x \leq
    1/3$. It follows that $v'_{\beta}(x) \leq 0$ when $\beta=2.4$
    and $0 \leq x \leq 1/3$. As $v_{\beta}(0)=0$, we deduce that
    $v_{\beta}(x) \leq 0$ when $\beta=2.4$
    and $0 \leq x \leq 1/3$. This means that when $\beta=2.4$,
    $u_{\beta}(x)$ is a decreasing function for $0 \leq x \leq
    1/3$.

    Our discussions above combined with direct calculations now imply that
\begin{equation*}
    c_{1/2} \geq \min \Big (\frac {2}{1+2.4}u^{1/2}_{2.4}(1/11) \approx 0.9001, \min_{1 \leq k \leq 10}s_k \Big ) \geq 0.9.
\end{equation*}
   This completes the proof of Theorem \ref{thm6}.

\section{Another proof of Theorem \ref{thm7}}
\label{sec 5} \setcounter{equation}{0}

    By H\"older's inequality, we have
\begin{equation*}
  \Big(\sum^{n}_{k=1}\big(k^{\alpha}-(k-1)^{\alpha}\big
)a_k \Big )^p \leq
\Big(\sum^{n}_{k=1}W^{1/(1-p)}_k\Big)^{p-1}\Big(\sum^{n}_{k=1}W_k\big(k^{\alpha}-(k-1)^{\alpha}
\big )^pa^p_k \Big),
\end{equation*}
   where $(W_k)$ is a sequence to be determined later. It follows
   that
\begin{eqnarray*}
  && \sum^{\infty}_{n=1}\Big ( \frac {1}{n^{\alpha}}\sum^{n}_{k=1}\big(k^{\alpha}-(k-1)^{\alpha}\big )a_k \Big
   )^p \\
   &\leq & \sum^{\infty}_{k=1}a^p_k\big(k^{\alpha}-(k-1)^{\alpha}\big
)^pW_k\sum^{\infty}_{n=k}\frac {1}{n^{\alpha p}}\Big (
\sum^{n}_{i=1}W^{1/(1-p)}_i \Big )^{p-1}.  \nonumber
\end{eqnarray*}

   We now choose $W_k$ to be
\begin{equation*}
    W_k=\big(k^{\alpha}-(k-1)^{\alpha} \big
)^{-p}\Big ( \sum^{\infty}_{i=k}i^{\gamma} \Big )^{-1},
\hspace{0.1in} \gamma=-\frac {\alpha p-1}{p}-1.
\end{equation*}

     We now note the following inequality (\cite[(30)]{Be1}) for $k \geq 1, \gamma>1$:
\begin{equation}
\label{5.01}
     (k^{\gamma}-(k-1)^{\gamma})\sum^{\infty}_{n=k}\frac
     {1}{n^{\gamma}} \leq \frac {\gamma}{\gamma-1}.
\end{equation}
     The above estimation implies that
\begin{equation*}
   W_k \geq \alpha^{-p}\Big (\frac {p}{\alpha p-1} \Big )^{-1}\Big
   (\frac {k^{\alpha}-(k-1)^{\alpha}}{\alpha} \Big )^{-p}\int^k_{k-1}x^{\alpha-1/p}dx.
\end{equation*}

     It follows that
\begin{eqnarray*}
     W_k &\geq& \alpha^{-p}\Big (\frac {p}{\alpha p-1} \Big )^{-1}
     \Big( \int^k_{k-1}x^{\alpha-1}dx \Big )^{-p}\int^k_{k-1}x^{\alpha-1/p}dx \\
     &\geq & \alpha^{-p}\Big (\frac {p}{\alpha p-1} \Big )^{-1}
     \Big( \int^k_{k-1}x^{\alpha-1/p-1}dx \Big )^{1-p},
\end{eqnarray*}
     where the last inequality above follows from H\"older's
     inequality. We then have
\begin{equation*}
    \Big (
\sum^{n}_{i=1}W^{1/(1-p)}_i \Big )^{p-1} \leq \alpha^{p}\Big
(\frac {p}{\alpha p-1} \Big )\Big (
\sum^{n}_{i=1}\int^{i}_{i-1}x^{\alpha-1/p-1}dx
\Big)^{p-1}=\alpha^{p}\Big (\frac {p}{\alpha p-1} \Big
)^pn^{(\alpha-1/p)(p-1)}.
\end{equation*}

     We then deduce that
\begin{eqnarray*}
   && \big(k^{\alpha}-(k-1)^{\alpha}\big
)^pW_k\sum^{\infty}_{n=k}\frac {1}{n^{\alpha p}}\Big (
\sum^{n}_{i=1}W^{1/(1-p)}_i \Big )^{p-1} \\
 & \leq & \alpha^{p}\Big (\frac {p}{\alpha p-1} \Big )^p\Big (\sum^{\infty}_{n=k}n^{-1-\alpha+1/p} \Big )^{-1}
 \sum^{\infty}_{n=k}\frac {n^{(\alpha-1/p)(p-1)}}{n^{\alpha p}}=\Big (\frac {\alpha p}{\alpha p-1} \Big )^p.
\end{eqnarray*}
   This now completes the proof of Theorem \ref{thm7}.

\section{Weighted Remainder Form of Hardy-type Inequalities }
\label{sec 6} \setcounter{equation}{0}
   In this section we study the weighted remainder form of Hardy-type inequalities in general. Let $(\lambda_n)$ be a positive sequence satisfying
   $\sum^{\infty}_{n=1}\lambda_n<+\infty$. We set in this section
$\Lambda_n=\sum^{\infty}_{k=n}\lambda_k$ and consider inequality
\eqref{6.1}. As we mentioned earlier, our goal is to find
conditions on the $\lambda_n$'s
   so that inequality \eqref{6.1} (or its reverse) can hold under
   these conditions. Our approaches in this section follow closely
   the approaches used in \cite{G1}, \cite{G6} and \cite{G7}.
   We first let $N$ be a large integer and for $1 \leq n \leq N$, we set $S_n=\sum^{N}_{k=n}\lambda_ka_k$ and
\begin{equation}
\label{6.3}
     A_n=\frac {\sum^{N}_{k=n}\lambda_ka_k}{\Lambda_n}.
\end{equation}

     It follows from \cite[(2.6)]{G} that we have for $0 < p<1,
   1/p+1/q=1$, $1 \leq k \leq N$,
\begin{equation}
\label{6.3'}
    \mu_kS^{1/p}_k-(\mu_k^q-\eta^q_k)^{1/q}S^{1/p}_{k+1} \leq \eta_k
    \lambda^{1/p}_ka^{1/p}_{k},
\end{equation}
    where $\mu^q_k > \eta^q_k \geq 0$ and the above inequality reverses when $p>1$. Here we define $S_{N+1}=0$.
    Due to similarities, we shall suppose $0 < p<1$ here and
    summing the above inequality
    from $k=1$ to $N$ leads to
\begin{equation*}
   \mu_1S^{1/p}_1+\sum^N_{k=2}\Big(\mu_k-(\mu_{k-1}^q-\eta^q_{k-1})^{1/q}
   \Big )S^{1/p}_k \leq
   \sum^N_{n=1}\eta_n
    \lambda^{1/p}_na^{1/p}_{n}.
\end{equation*}
    We now set $\eta_i=\lambda^{-1/p}_i$ together with a change of
    variables $\mu_i \rightarrow \mu_i\eta_i$ to recast the above
    inequality as
\begin{equation*}
   \frac {\mu_1S^{1/p}_1}{\lambda^{1/p}_1}+\sum^N_{k=2}\Big(\frac {\mu_k}{\lambda^{1/p}_k}-\frac
   {(\mu_{k-1}^q-1)^{1/q}}{\lambda^{1/p}_{k-1}}
   \Big )S^{1/p}_k \leq
   \sum^N_{n=1}a^{1/p}_{n}.
\end{equation*}
    We further set $\mu^q_i-1=\nu_i$ and make a further change of variables: $p \rightarrow 1/p$ to recast the above
    inequality as
\begin{equation*}
   \frac {(1+\nu_1)^{1-p}S^{p}_1}{\lambda^{p}_1}+\sum^N_{k=2}\Big(\frac {(1+\nu_k)^{1-p}}{\lambda^{p}_k}-\frac
   {\nu^{1-p}_{k-1}}{\lambda^{p}_{k-1}}
   \Big )S^{p}_k  \leq
   \sum^N_{n=1}a^{p}_{n}.
\end{equation*}
     We now set $\nu_i=\sum^{\infty}_{n={i+1}}w_i/w_i$, where $w_n$'s are positive parameters, to recast
     the above inequality as
\begin{equation*}
    \frac {w^{p-1}_1}{\lambda^{p}_1}\Big (\sum^{\infty}_{i=1}w_i\Big )^{1-p}\Lambda^p_1A^p_1+\sum_{n=2}^{N}
    \Big(\sum^{\infty}_{k=n}w_k \Big )^{-(p-1)}\Big( \frac {w_n^{p-1}}{\lambda^p_n}-\frac {w_{n-1}^{p-1}}{\lambda^p_{n-1}} \Big )
    \Lambda^p_n A_n^{p} \leq \sum_{n=1}^N
    a_n^{p}.
\end{equation*}

   By a change of variables $w_n \rightarrow \lambda_nw^{1/(p-1)}_n$,
   we can recast the above inequality as
\begin{equation*}
    \frac {w_1}{\lambda_1}\Big (\sum^{\infty}_{i=1}\frac {\lambda_iw^{1/(p-1)}_i}{\Lambda_1}\Big )^{1-p}\Lambda_1A^p_1+\sum_{n=2}^{N}
    \Big( \frac {\sum^{\infty}_{k=n}\lambda_kw^{1/(p-1)}_k}{\Lambda_n} \Big )^{-(p-1)}\Big( \frac {w_n}{\lambda_n}-\frac {w_{n-1}}{\lambda_{n-1}} \Big )
    \Lambda_n A_n^{p} \leq \sum_{n=1}^N  a_n^{p}.
\end{equation*}
   With another change of variables, $w_n/w_{n-1} \rightarrow b_n$ with $w_0=1$,
   we can further recast the above inequality as
\begin{eqnarray}
\label{5.4}
   &&\frac {b_1}{\lambda_1}\Big( \frac {\sum^{\infty}_{k=1}\lambda_k\prod^{k}_{i=1}b^{1/(p-1)}_i}{\Lambda_n} \Big )^{-(p-1)}\Lambda_1A^p_1 \\
   &&+ \sum_{n=2}^{N}
    \Big( \frac {\sum^{\infty}_{k=n}\lambda_k\prod^{k}_{i=n}b^{1/(p-1)}_i}{\Lambda_n} \Big )^{-(p-1)}\Big( \frac {b_n}{\lambda_n}-\frac {1}{\lambda_{n-1}} \Big )
    \Lambda_n A_n^{p}  \leq \sum_{n=1}^N  a_n^{p}.  \nonumber
\end{eqnarray}
    We now choose the $b_n$'s to satisfy:
\begin{equation*}
     \sum^{\infty}_{k=n}\lambda_k\prod^{k}_{i=n}b^{1/(p-1)}_i=
     \frac {p}{p-L}\Lambda_n.
\end{equation*}
     From this we solve the $b_n$'s to get
\begin{equation*}
     b_n=\Big (1-\frac {L}{p}\frac {\lambda_n}{\Lambda_n} \Big )^{1-p}.
\end{equation*}

    Upon requiring $\Lambda_n(b_n/\lambda_n-1/\lambda_{n-1}) \geq 1-L/p$ (with $1/\lambda_0=0$) and letting $N \rightarrow +\infty$,
    we deduce easily from \eqref{5.4} the $p>1$ (and $0<p<1$)
    cases of
    the following
\begin{theorem}
\label{thm3.1}
  Let $p \neq 0$ be fixed and $a_n > 0$. Let $L$ be a number satisfying $L<p$ when $p>0$ and $L>p$ when
  $p<0$. Suppose that $\lim_{n \rightarrow
  \infty}\Lambda_{n+1}(\sum^{\infty}_{k=n+1}\lambda_ka_k/\Lambda_{n+1})^p/\lambda_{n}=0$ when $p<0$.
  When $p > 1$ or $p<0$, if (with
  $\Lambda_0/\lambda_0=1$) for $n \geq 1$,
\begin{equation}
\label{6.6'}
    \frac {\Lambda_{n-1}}{\lambda_{n-1}} \leq \frac
    {\Lambda_n}{\lambda_n}  \Big (1- \frac
    {L\lambda_n}{p\Lambda_n} \Big )^{1-p}+\frac {L}{p},
\end{equation}
    then inequality \eqref{6.1} holds when $p>1$ or $p<0$. If
    the reversed inequality above holds when $0<p<1$, then the
    reversed inequality of \eqref{6.1} also holds.
\end{theorem}
   The case $p<0$ of Theorem \ref{thm3.1}
   follows from the same arguments above staring from inequality
   \eqref{6.3'}, as it still holds for $p<0$, except this time we substitute $S_n$ by
   $\sum^{\infty}_{k=n}\lambda_ka_k$ and $A_n$ by
   $\sum^{\infty}_{k=n}\lambda_ka_k/\Lambda_n$. In this case, we may assume $\sum^{\infty}_{k=n}\lambda_ka_k/\Lambda_n<+\infty$,
   for otherwise, inequality
  \eqref{6.1} holds automatically.

    On taking Taylor expansion of the right-hand side expression of \eqref{6.6'}, we deduce easily from Theorem \ref{thm3.1} the following
\begin{cor}
\label{cor3.1}
  Let $p \neq 0$ be fixed and $a_n > 0$. Let $L$ be a number satisfying $L<p$ when $p>0$ and $L>p$ when
  $p<0$. Suppose that $\lim_{n \rightarrow
  \infty}\Lambda_{n+1}(\sum^{\infty}_{k=n+1}\lambda_ka_k/\Lambda_{n+1})^p/\lambda_{n}=0$ when $p<0$.
  When $p > 1$ or $p<0$, if (with
  $\Lambda_0/\lambda_0=1$) for $n \geq 1$,
\begin{equation}
\label{6.6}
    L \geq  \frac {\Lambda_{n-1}}{\lambda_{n-1}}-\frac {\Lambda_{n}}{\lambda_{n}},
\end{equation}
    then inequality \eqref{6.1} holds when $p>1$ or $p<0$. If
    the reversed inequality above holds when $0<p<1$, then the
    reversed inequality of \eqref{6.1} also holds.
\end{cor}

    We now give an improvement of the above result:
\begin{theorem}
\label{thm3.2}
  Let $p \neq 0$ be fixed and $a_n > 0$. Let $L$ be a number satisfying $L<p$ when $p>0$ and $L>p$ when $p<0$. Suppose that $\lim_{n \rightarrow
  \infty}\Lambda_{n+1}(\sum^{\infty}_{k=n+1}\lambda_ka_k/\Lambda_{n+1})^p/\lambda_{n}=0$ when $p<0$.
  When $p \geq 1$ or $p<0$, if (with
  $\Lambda_0/\lambda_0=1$) for $n \geq 1$, inequality \eqref{6.6}
  holds,
    then for $p \geq 1$,
\begin{equation*}
   \sum^{\infty}_{n=1}\Big (\sum^{\infty}_{k=n}\frac
   {\lambda_ka_k}{\Lambda_n}\Big )^p \leq \frac {p}{p-L}
    \sum^{\infty}_{n=1}a_n \Big ( \sum^{\infty}_{k=n}\frac
   {\lambda_ka_k}{\Lambda_n}  \Big )^{p-1}.
\end{equation*}
   The above inequality reverses when $p<0$. When $0<p \leq 1$, the reversed
   inequality above also holds if the reversed inequality
   \eqref{6.6} holds for all $n \geq 1$.
\end{theorem}
\begin{proof}
    We consider the cases $p \geq 1$ and $0<p \leq 1$ first. Due to similarities, we assume $p \geq 1$ here. We let $N$ be a large integer and
    start with the inequality $x^p-px+p-1 \geq 0$, valid for
   $x>0, p \geq 1$ or $p<0$ with the reversed inequality being valid for
  $x>0, 0<p \leq 1$.
  On setting $x=A_{n+1}/A_n$, $1 \leq n \leq N$ with $A_{N+1}=0$, where $A_n, 1 \leq n \leq N$ is defined as in \eqref{6.3}, we obtain
\begin{equation}
\label{3.4}
   A^p_{n+1}+(p-1)A^p_n \geq pA_{n+1}A^{p-1}_n.
\end{equation}
    Note that
\begin{equation*}
   A_{n+1}=\frac {\Lambda_nA_n}{\Lambda_{n+1}}-\frac
   {\lambda_na_n}{\Lambda_{n+1}}.
\end{equation*}
    Substituting this expression of $A_{n+1}$ on the right-hand
    side of \eqref{3.4}, we obtain after some simplifications that
\begin{equation*}
   \Big ( \frac {\Lambda_n}{\lambda_n}+p-1 \Big ) A^p_n-\Big (\frac {\Lambda_n}{\lambda_n}-1 \Big )A^p_{n+1} \leq pa_nA^{p-1}_n .
\end{equation*}
    Summing the above inequality from $n=1$ to $N$, we obtain
\begin{equation}
\label{6.9}
    \sum^N_{n=1}\Big ( \frac {\Lambda_n}{\lambda_n}
   -\frac {\Lambda_{n-1}}{\lambda_{n-1}}+p \Big ) A^p_n \leq p\sum^{N}_{n=1}a_nA^{p-1}_n .
\end{equation}
    The assertion of the
    theorem for the cases $p \geq 1$ now follows easily from the case $N \rightarrow +\infty$ of the above
    inequality and inequality \eqref{6.6}.

    The case $p<0$ of the assertion of the theorem follows from the same arguments above, except this time we substitute $A_n$ by
   $\sum^{\infty}_{k=n}\lambda_ka_k/\Lambda_n$. In this case, we may assume $\sum^{\infty}_{k=n}\lambda_ka_k/\Lambda_n<+\infty$,
   for otherwise, the assertion of the theorem holds automatically.
\end{proof}
     When $\sum^{\infty}_{n=1}a^p_n< +\infty$ and that $\sum^{\infty}_{n=1}\Big (\sum^{\infty}_{k=n}\lambda_ka_k/\Lambda_n\Big )^p <+\infty$,
     then by H\"older's inequality, we have for $p>1$,
\begin{equation*}
    \sum^{\infty}_{n=1}a_n \Big ( \sum^{\infty}_{k=n}\frac
   {\lambda_ka_k}{\Lambda_n}  \Big )^{p-1} \leq
   \Big(\sum^{\infty}_{n=1}a^p_n \Big )^{1/p}\Big (\sum^{\infty}_{n=1}\Big (\sum^{\infty}_{k=n}\frac
   {\lambda_ka_k}{\Lambda_n}\Big )^p \Big )^{1/q}.
\end{equation*}
   with the above inequality reversed when $0 \neq p<1$ and from
   which one easily deduces the assertion of Corollary
   \ref{cor3.1}. Note that when $p>0$, one can also deduce the assertion of Corollary
   \ref{cor3.1} without assuming $\sum^{\infty}_{n=1}\Big (\sum^{\infty}_{k=n}\lambda_ka_k/\Lambda_n\Big )^p
   <+\infty$. Since one can start with \eqref{6.9}, repeat the
   argument above and then let $N \rightarrow +\infty$.

    We now study the following so called weighted remainder form of Carleman-type inequality, corresponding to the limiting case $p \rightarrow +\infty$ of
inequality \eqref{6.1} (after a change of variables $a^p_n
\rightarrow a_n$):
\begin{equation}
\label{6.8}
   \sum^{\infty}_{n=1}\Big( \prod^{\infty}_{k=n}a^{\lambda_k/\Lambda_n}_k \Big ) \leq E\sum^{\infty}_{n=1}a_n.
\end{equation}
 This is first studied by Pe\v cari\'c and
Stolarsky in \cite[Sect 3]{P&S}.
   Our starting point is the following result of Pe\v cari\'c and Stolarsky
\cite[(3.5)]{P&S}, which is an outgrowth of Redheffer's approach
in \cite{R1}:
\begin{equation}
\label{5.1}
   \sum^N_{n=1}\Lambda_n(b_n-1)G_n+G_1\Lambda_1-\Lambda_{N+1}G_{N+1}
\leq \sum^N_{n=1}\lambda_na_nb^{\Lambda_n/\lambda_n}_n,
\end{equation}
  where $N$ is a large integer, ${\bf b}$ is any positive sequence and
\begin{equation*}
   G_n=\prod^{\infty}_{k=n}a^{\lambda_k/\Lambda_n}_k.
\end{equation*}
   We now make a change of variables $\lambda_na_nb^{\Lambda_n/\lambda_n}_n \rightarrow a_n$ to recast
  inequality \eqref{5.1} as
\begin{eqnarray*}
  && \sum^N_{n=1}\Lambda_n(b_n-1)
   \Big( \prod^{\infty}_{k=n}\lambda^{-\lambda_k/\Lambda_n}_k \Big )\Big( \prod^{\infty}_{k=n}b^{-\Lambda_k/\Lambda_n}_k \Big
   )G_n+G_1\Lambda_1\Big( \prod^{\infty}_{k=1}\lambda^{-\lambda_k/\Lambda_n}_k \Big )\Big( \prod^{\infty}_{k=1}b^{-\Lambda_k/\Lambda_n}_k \Big
   ) \\
   &&-G_{N+1}\Lambda_{N+1}\Big( \prod^{\infty}_{k=N+1}\lambda^{-\lambda_k/\Lambda_{N+1}}_k \Big )\Big( \prod^{\infty}_{k=N+1}b^{-\Lambda_k/\Lambda_{N+1}}_k \Big
   )
\leq \sum^N_{n=1}a_n.
\end{eqnarray*}

   Now, a further change of variables $b_n \rightarrow
   \lambda_{n-1}b_n/\lambda_n$ with $\lambda_0>0$ an arbitrary number allows us to recast the above
   inequality as
\begin{equation}
\label{5.2}
  \frac {\Lambda_1b_1G_1}{\lambda_1}\prod^{\infty}_{k=1}b^{-\Lambda_k/\Lambda_n}_k
  + \sum^N_{n=2}\Lambda_n\Big(\frac {b_n}{\lambda_n}-\frac {1}{\lambda_{n-1}}\Big)G_n \prod^{\infty}_{k=n}b^{-\Lambda_k/\Lambda_n}_k
  -\frac {\Lambda_{N+1}G_{N+1}}{\lambda_N} \prod^{\infty}_{k=N+1}b^{-\Lambda_k/\Lambda_{N+1}}_k
\leq \sum^N_{n=1}a_n.
\end{equation}

   If we now choose the values of $b_n$'s so that
   $\prod^{\infty}_{k=n}b^{-\Lambda_k/\Lambda_n}_k=e^{-M}$, we
   then solve the $b_n$'s to get $b_n=e^{M\lambda_n/\Lambda_n}$ and upon substituting these
   values for $b_n$'s we obtain via \eqref{5.2}:
\begin{equation}
\label{5.2'}
   \frac {\Lambda_1}{\lambda_1}e^{M\lambda_1/\Lambda_1}G_1
   +\sum^N_{n=2}\Lambda_n\Big(\frac {e^{M\lambda_n/\Lambda_n}}{\lambda_n}-\frac {1}{\lambda_{n-1}}\Big) G_n
   -\frac {\Lambda_{N+1}}{\lambda_N}G_{N+1} \leq e^M\sum^N_{n=1}a_n.
\end{equation}

    We immediately deduce from \eqref{5.2'} the following
\begin{theorem}
\label{thm0}
   Suppose that $\lim_{n \rightarrow
   \infty}\Lambda_{n+1}G_{n+1}/\lambda_n = 0$ and that (with
   $\Lambda_0/\lambda_0=1$)
\begin{equation*}
  M=\sup_{n \geq 1}\frac
    {\Lambda_n}{\lambda_n}\log \Big(\frac {\Lambda_{n-1}/\lambda_{n-1}}{\Lambda_n/\lambda_n} \Big ) < +\infty,
\end{equation*}
    then inequality \eqref{6.8} holds with $E=e^M$.
\end{theorem}

    We note that
\begin{equation*}
 \log \Big(\frac {\Lambda_{n-1}/\lambda_{n-1}}{\Lambda_n/\lambda_n}
\Big )=\log \Big(1+\frac
{\Lambda_{n-1}/\lambda_{n-1}-\Lambda_n/\lambda_n}{\Lambda_n/\lambda_n}
\Big ) \leq \frac
{\Lambda_{n-1}/\lambda_{n-1}-\Lambda_n/\lambda_n}{\Lambda_n/\lambda_n}.
\end{equation*}
    It follows from this and Theorem \ref{thm0} that we have the
    following
\begin{cor}
\label{cor2.0}
   Suppose that $\lim_{n \rightarrow
   \infty}\Lambda_{n+1}G_{n+1}/\lambda_n = 0$ and that (with
   $\Lambda_0/\lambda_0=1$)
\begin{equation*}
  M=\sup_{n \geq 1} \Big ( \frac {\Lambda_{n-1}}{\lambda_{n-1}}-\frac {\Lambda_{n}}{\lambda_{n}} \Big ) < +\infty,
\end{equation*}
  then inequality \eqref{6.8} holds with $E=e^M$.
\end{cor}

    We now consider another choice for the $b_n$'s in \eqref{5.2} by
    setting
    $b_n=e^{(\Lambda_{n-1}/\lambda_{n-1}-\Lambda_{n}/\lambda_{n})/(\Lambda_{n}/\lambda_{n})}$
    with $\Lambda_0/\lambda_0=1$
    and it follows from this and \eqref{5.2} that
\begin{eqnarray*}
  && \sum^N_{n=1}\Big(\frac {\Lambda_ne^{(\Lambda_{n-1}/\lambda_{n-1}-\Lambda_{n}/\lambda_{n})/(\Lambda_{n}/\lambda_{n})}}{\lambda_n}
   -\frac {\Lambda_{n-1}}{\lambda_{n-1}}+1\Big)G_n
   e^{-\sum^{\infty}_{k=n}\frac {\lambda_k}{\Lambda_n}(\frac {\Lambda_{k-1}}{\lambda_{k-1}}-\frac {\Lambda_{k}}{\lambda_{k}})} \\
   && -\frac {\Lambda_{N+1}G_{N+1}}{\lambda_N}
   e^{-\sum^{\infty}_{k=N+1}\frac {\lambda_k}{\Lambda_{N+1}}(\frac {\Lambda_{k-1}}{\lambda_{k-1}}-\frac {\Lambda_{k}}{\lambda_{k}})}\leq
\sum^N_{n=1}a_n,
\end{eqnarray*}
   from which we deduce the following
\begin{cor}
\label{cor2.00}
  Suppose that $\lim_{n \rightarrow
   \infty}\Lambda_{n+1}G_{n+1}/\lambda_n = 0$ and that (with
   $\Lambda_0/\lambda_0=1$)
\begin{equation*}
  M=\sup_{n \geq 1} \sum^{\infty}_{k=n}\frac {\lambda_k}{\Lambda_n}(\frac {\Lambda_{k-1}}{\lambda_{k-1}}-\frac {\Lambda_{k}}{\lambda_{k}}) < +\infty,
\end{equation*}
  then inequality \eqref{6.8} holds with $E=e^M$.
\end{cor}
    Note that the above corollary also implies Corollary
    \ref{cor2.0}. We now consider some applications of our results above.
     When $\lambda_n=n^{\alpha}-(n+1)^{\alpha}$, $-1 \leq
     \alpha <0$, by Lemma 1 (and property (iv) of $f_{\alpha}(x)$ defined there) of
     \cite{Be1}, we have for $n \geq 1$,
\begin{equation}
\label{6.11}
     \frac {n^{\alpha}}{n^{\alpha}-(n+1)^{\alpha}}-\frac
     {(n+1)^{\alpha}}{(n+1)^{\alpha}-(n+2)^{\alpha}} \leq \frac
     {1}{\alpha}.
\end{equation}
     One can show also easily that $1-\alpha \geq 2^{-\alpha}$ for $-1 \leq \alpha
     <0$ and that for fixed $-1 \leq \alpha<0$. Moreover, we have
     $(n+1)^{\alpha}G_{n+1}/(n^{\alpha}-(n+1)^{\alpha}) \leq
     (n+1)G_{n+1}/(-\alpha)$, so that it follows from Corollary \ref{cor2.0} that we
     have the following
\begin{cor}
\label{cor2.1}
   Let $-1 \leq \alpha <0$ and assume that $\lim_{n \rightarrow
   \infty}nG_{n}=0$, then
\begin{equation}
\label{2.4}
   \sum^{\infty}_{n=1}\Big( \prod^{\infty}_{k=n}a^{(k^{\alpha}-(k+1)^{\alpha})/n^{\alpha}}_k \Big ) \leq e^{1/\alpha}\sum^{\infty}_{n=1}a_n.
\end{equation}
    The constant is best possible.
\end{cor}
     By taking $a_n=n^{-1-\epsilon}$ with $\epsilon \rightarrow
     0^+$, one shows that the constant in \eqref{2.4} is indeed best
     possible. We note that inequality \eqref{6.11} is reversed when $\alpha \leq
     -1$ and we also have $1-\alpha \leq 2^{-\alpha}$ when $\alpha
     \leq -1$ and it follows from Corollary \ref{cor3.1} that this
     gives another proof of the case $r \leq -1, 0<p<1$ of inequality
     \eqref{2.02}. We point out that similar to the treatment in inequality
     \eqref{1.7}, one can show the case $r \leq -1, 0<p<1$ of inequality
     \eqref{2.02} also follows from inequality \eqref{bound1} and
     the validity of inequality \eqref{8} for $\alpha \geq 2, p<0$.

     When $\lambda_n=n^{\alpha}$, $\alpha < -1$, it follows from \eqref{5.01} with $\gamma=-\alpha$, we have for $n \geq 1$,
\begin{equation*}
     \frac {\sum^{\infty}_{i=n}i^{\alpha}}{n^{\alpha}}-\frac
     {\sum^{\infty}_{i=n+1}i^{\alpha}}{(n+1)^{\alpha}} \geq \frac
     {1}{\alpha+1}.
\end{equation*}
     Note that the case $k=1$ of \eqref{5.01} with $\gamma=-\alpha$ also implies the reversed inequality \eqref{6.6}
     when $n=1$ and it follows from Corollary \ref{cor3.1} that we have the
     following
\begin{cor}
\label{cor6.5}
   Let $\alpha <-1$ and $a_n>0$, then for $0<p<1$,
\begin{equation}
\label{6.14}
   \sum^{\infty}_{n=1}\Big( \sum^{\infty}_{k=n}\frac {k^{\alpha}a_k}{\sum^{\infty}_{i=n}i^{\alpha}} \Big )^p \geq
   \Big ( \frac {(1+\alpha)p}{(1+\alpha)p-1} \Big )^p\sum^{\infty}_{n=1}a^p_n.
\end{equation}
    The constant is best possible.
\end{cor}
    By taking $a_n=n^{-1/p-\epsilon}$ with $\epsilon \rightarrow
     0^+$, one shows that the constant in \eqref{6.14} is indeed best
     possible. We point out here that similar to the treatment for
     the $p>1$ case of Theorem \ref{thm1} given in Section
     \ref{sec 1}, we can recast inequality \eqref{6.14} via the duality principle
     as (by a change of variable $\alpha \rightarrow -\alpha$),
\begin{equation*}
    \sum^{\infty}_{n=1}\Big( \sum^{n}_{k=1}
    \frac {\alpha}{(\alpha-1)(k^{\alpha}-(k-1)^{\alpha})\sum^{\infty}_{i=k}i^{-\alpha}}\frac {(k^{\alpha}-(k-1)^{\alpha})a_k}{n^{\alpha}} \Big )^p \leq
   \Big ( \frac {\alpha p}{\alpha p-1} \Big )^p\sum^{\infty}_{n=1}a^p_n.
\end{equation*}
    Here we have $\alpha > 1$ and $p<0$. It is then easy to see that the above inequality (hence inequality \eqref{6.14}) also follows from
    inequality \eqref{5.01} and the $p<0$ case of inequality
    \eqref{1.5}.

\vskip0.1in
\noindent {\bf Acknowledgements.}
  The author is supported by a research fellowship from an Academic
Research Fund Tier 1 grant at Nanyang Technological University for
this work.

\end{document}